\documentclass[a4paper,11pt]{article}
\usepackage{amsmath,amssymb,amscd,amsthm,txfonts,url}
\usepackage[backref]{hyperref}
\usepackage{pdfsync}
\usepackage[T1]{fontenc}
\usepackage[utf8]{inputenc}
\usepackage{xypic}

\newtheorem{theorem}{Theorem}[subsection]

\newtheorem{corollary}[theorem]{Corollary}

\numberwithin{equation}{subsection}

\theoremstyle{remark}
\newtheorem{remark}[theorem]{Remark}

\theoremstyle{definition}

\def\I{\mathrm{I}}
\def\R{\mathrm{R}}
\def\GL{\mathrm{GL}}
\def\F{\mathbb{F}}

\def\calC{{\mathcal{C}}}
\def\calN{{\mathcal{N}}}
\def\car{\mathfrak{car}}

\def\calH{{\mathcal{H}}}

\def\calD{{\mathcal{D}}}

\def\calS{{\mathcal{S}}}

\def\Q{{\overline{\mathbb{Q}}_\ell}}
\def\B{\mathrm{B}}
\def\pr{\mathrm{pr}}
\def\Ind{\mathrm{Ind}}
\def\Res{\mathrm{Res}}
\def\R{\mathrm{R}}
\def\P{\mathrm{P}}

\def\GL{\mathrm{GL}}

\def\GL{\mathrm {GL}}

\def\B{\mathrm{B}}

\def\Tr{{\rm Tr}}
\begin{document}

\title{Geometrization of the two orthogonality formulas for Green functions}

\author{ G\'erard Laumon
\\ {\it Universit\'e Paris-saclay, LMO CNRS UMR 8628}
\\{\tt gerard.laumon@u-psud.fr } \and Emmanuel Letellier \\ {\it
  Universit\'e Paris Cit\'e, IMJ-PRG CNRS UMR 7586} \\ {\tt
  emmanuel.letellier@imj-prg.fr} }

\pagestyle{myheadings}
\maketitle

\begin{abstract}The Green functions were first introduced by Green to compute the character table of $\GL_n(\F_q)$ in 1955. They were later generalized by Deligne and Lusztig for an arbitrary finite group of Lie type $G(\F_q)$ using $\ell$-adic cohomological methods (1976). They proved that these Green functions satisfy an orthogonality relation (we call the first orthogonality relation). Ten years later Kawanaka proved that they satisfy an other orthogonality relation (we call the second orthogonality relation). In this notes, we explain how the main results of our paper \cite{LL} provide a geometric understanding of these two orthogonality relations and how we can see geometrically that the  two orthogonality relations are in fact equivalent. In this geometric approach the language of stacks is essential.
\end{abstract}

\tableofcontents

\maketitle

\section{Introduction}Let $G$ be a connected reductive algebraic group defined over a finite field $k_o$ with corresponding Frobenius $F:G\rightarrow G$ and denote by $G_{\rm uni}$ the closed subset of unipotent elements of $G$. For any $F$-stable maximal torus $T$ of $G$ with a  Borel subgroup $B=TU$ (not necessarily $F$-stable), Deligne and Lusztig \cite{DLu}\cite{Lu} (see also \cite{DM}) defined   $Q_{T^F}^{G^F}:G_{\rm uni}^F\rightarrow\Q$ (\emph{Green function}) by
$$
Q_{T^F}^{G^F}(u)=\frac{1}{|T^F|}\sum_i(-1)^i\Tr\left(u, H_c^i(\mathcal{L}^{-1}(U),\Q)\right),
$$
where $\mathcal{L}:G\rightarrow G, x\mapsto x^{-1}F(x)$ is the Lang map and where the action of $G^F$ on $H^i(\mathcal{L}^{-1}(U),\Q)$ is induced by the left multiplication of $G^F$ on $\mathcal{L}^{-1}(U)$. 

It is well-known that the Green functions do not depend on the choice of the Borel subgroup $B$ containing $T$. 
\bigskip

The (first) orthogonality formula 

\begin{equation}
\frac{1}{|G^F|}\sum_{u\in G_{\rm uni}^F}Q_{T^F}^{G^F}(u)Q_{T'{^F}}^{G^F}(u)=\frac{|N_{G^F}(T,T')|}{|T^F|\, |T'^{F}|}
\label{ortho1}\end{equation}
where $N_{G^F}(T,T')=\{g\in G^F\,|\, g^{-1}Tg=T'\}$, was proved by Deligne-Lusztig \cite{DLu}.
\bigskip

Fix now a maximally split $F$-stable maximal torus $T$ of $G$ and let $B$ be an $F$-stable Borel subgroup of $G$ containing $T$. Denote by $W=N_G(T)/T$ the Weyl group of $G$ with respect to $T$.  The Frobenius $F$ acts on $W$ and the set $H^1(F,W)$ of $F$-conjugacy classes of $W$ parametrizes the set of $G^F$-conjugacy classes of $F$-stable maximal tori of $G$. For $[w]\in H^1(F,W)$ we will denote by $T_w$ a representative of the corresponding $G^F$-conjugacy class of $F$-stable maximal tori. Then The Frobenius $F$ on $T_w$ corresponds to the Frobenius $wF$ on $T$ and so $T_w^F\simeq T^{wF}$.
\bigskip

In the case where $G=\GL_n$, the second orthogonality formula for Green functions reads, for $u,v\in G_{\rm uni}^F$,

$$
\frac{1}{|W|}\sum_{w\in W}|T_w^F|\, Q_{T_w^F}^{G^F}(u) Q_{T_w^F}^{G^F}(v)=\begin{cases} |C_G(u)^F|&\text{ if }u\text{ is }G^F\text{-conjugate to } v,\\
0&\text{ otherwise.}\end{cases}
$$
\bigskip

\begin{remark}When $G=\GL_n$ and $F$ is the standard Frobenius that raises coefficients of matrices to their $q$-th power, then we can take the maximal torus of diagonal matrices for$T$ and identify $W$ with the symmetric group $S_n$. The Frobenius $F$ acts trivially on $W=S_n$ and so $H^1(F,W)$ is the set of conjugacy classes of $S_n$. Then for $w\in S_n$ and $u\in G_{\rm uni}^F$, we have \cite{Sp1}

\begin{equation}
Q_{T_w^F}^{G^F}(u)=Q^\mu_\rho(q)
\end{equation}
where $\rho$ is the partition of $n$ corresponding to the decomposition of $w$ as a product of disjoint cycles and $\mu$ is the partition given by the size of the Jordan blocks of $u$ and where $Q^\mu_\rho(q)$ is the Green polynomial as defined in \cite{Green}\cite[III, \S 7]{macdonald}. Then the  first orthogonality relation is \cite[III, (7.10)]{macdonald} and the second one is \cite[III, (7.9)]{macdonald}. 
\end{remark}
\bigskip

The second orthogonality relation for Green functions has been generalized by Kawanaka \cite{Kawanaka} to an arbitrary connected reductive group $G$. It is written using the framework of Springer correspondence. We can rephrase it as follows. 
\bigskip

Denote by $\P^F$ the projector on \emph{uniform functions} (see \S\ref{DLu}) and let $\Delta_{\rm uni}^F:G_{\rm uni}^F\times G_{\rm uni}^F\rightarrow \Q$ be the function defined by 

\begin{equation}
\Delta_{\rm uni}^F (u,v)=\begin{cases}|C_G(u)^F|&\text{ if }u\text{ is }G^F-\text{conjugate to }v,\\0&\text{ otherwise.}\end{cases}
\label{Delta}\end{equation}
Then the second orthogonality relation \cite[(1.1.7)]{Kawanaka} reads

\begin{equation}
\frac{1}{|W|}\sum_{w\in W}|T_w^F|\,Q_{T_w^F}^{G^F}(u)\,Q_{T_w^F}^{G^F}(v)=(\P^F\times \P^F)(\Delta_{\rm uni}^F)(u,v)
\label{ortho2}\end{equation}
for all $u,v\in G_{\rm uni}^F$.
\bigskip

The aim of this note is to give a geometrical understanding of the two orthogonality formulas and to show that the two are equivalent in a simple way (which seems to be a new result).
\bigskip

{\bf Acknowledgments:} We thank George Lusztig for pointing out some slighty inaccurate references.

\section{Preliminaries}

In this section we review the results of \cite{LL}. Let $k$ be an algebraic closure of a finite field and let $\ell$ be a prime invertible in $k$. 
\bigskip

We fiw a connected reductive algebraic group over $k$, a maximal torus $T$ of $G$, and we denote by $N$ the normalizer of $T$ in $G$ and by $W$ be the Weyl group $N/T$.
\bigskip

We also choose a Borel subgroup $B=TU$ of $G$ containing $T$.
\bigskip

If $Z$ is a $k$-scheme on which a $k$-algebraic group $H$ acts, we denote by $[Z/H]$ the quotient stack of $Z$ by $H$. The classifying stack of $H$-torsors is then $\B(H):=[{\rm Spec}(k)/H]$. 

An $S$-point of $[Z/H]$, i.e. a $k$-morphism $S\rightarrow [Z/H]$, corresponds to a cartesian diagram 

$$
\xymatrix{E\ar[rr]\ar[d]&&Z\ar[d]\\
S\ar[rr]&&[Z/H]}
$$
where $E\rightarrow S$ is an $H$-torsor and $E\rightarrow Z$ is $H$-equivariant.

\subsection{Lusztig's induction and restriction functors}

We consider the geometric correspondence

$$
\xymatrix{[T/T]&&[B/B]\ar[rr]^p\ar[ll]_q&&[G/G]}
$$
where $q$ is induced by the projection $B\rightarrow T$ and $p$ by the inclusion $B\subset G$ and where the actions considered in the quotients are the conjugation actions. As the conjugation action of $T$ on $T$ is trivial we have $[T/T]\simeq T\times\B(T)$.
\bigskip

\begin{remark}If we put 

$$
X:=\{(x,gB)\in G\times G/B\,|\, g^{-1}xg\in B\}
$$
then $[B/B]$ is isomorphic to $[X/G]$ via the map $B\rightarrow X$, $x\mapsto (x,B)$, and $p$ is just the quotient by $G$ of the first projection $X\rightarrow G$.
\end{remark}
\bigskip

We have the corresponding Lusztig induction functor 

$$
\Ind:\mathcal{D}_c^b([T/T])\rightarrow\mathcal{D}_c^b([G/G]), \hspace{.5cm}K\mapsto p_!q^!(K)
$$
whose left adjoint is the restriction functor

$$
\Res:\mathcal{D}_c^b([G/G])\rightarrow\mathcal{D}_c^b([T/T]),\hspace{.5cm}K\mapsto q_!p^*(K).
$$

\begin{remark}Recall that these two functors preserves perverse sheaves by \cite{BY}. Moreover, the categories of perverse sheaves on $[T/T]$ and $T$ are equivalent by the functor $s^![n](n)$ where $n$ is the rank of $G$ and $s:T\rightarrow[T/T]$ is the quotient map. So in the case of perverse sheaves the functor $\Ind$ coincides with the classical geometric induction from perverse sheaves on $T$ to $G$-equivariant perverse sheaves on $G$ (which is considered in Lusztig's character sheaves theory).
\label{rem1}\end{remark}

We now consider the Stein factorisation of $p$, namely we consider the morphism of correspondences

\begin{equation}
\xymatrix{&&[B/B]\ar[rrd]^p\ar[lld]_q\ar[d]^{(q,p)}&&\\
[T/T]&&\calS:=[T/T]\times_\car[G/G]\ar[rr]^{\pr_2}\ar[ll]_{\pr_1}&&[G/G]}
\label{triangle}\end{equation}
where $\car:=T/\!/W$.

Put 
$$
\calN:=(q,p)_!\Q
$$
From the projection formulas, we see that

$$
\Ind(K)=\pr_{2*}\underline{\rm Hom}\left(\calN,\pr_1^!(K)\right),\hspace{1cm}\Res(K')=\pr_{1!}\left(\pr_2^*(K')\otimes\calN\right)
$$
for $K\in\calD_c^b([T/T])$ and $K'\in\calD_c^b([G/G])$.
\bigskip

It is crucial to notice that the $(q,p)$ is not small unlike the morphism $f:[B/B]\rightarrow T\times_\car [G/G]$ which can be used instead of $(q,p)$ when working with perverse sheaves by Remark \ref{rem1}. In fact the connection between $(q,p)$ and $f$ is as follows.
\bigskip

Denote by $\tau:[B/B]\rightarrow \B(T)$ the map making the following diagram cartesian

$$
\xymatrix{[B/U]\ar[d]\ar[rr]&&{\rm Spec}(k)\ar[d]\\
[B/B]\ar[rr]^\tau&&\B(T)}
$$
Then the map $(q,p)$ decomposes as 

$$
\xymatrix{[B/B]\ar[rr]^-{(\tau,{\rm id}_{[B/B]})}&&\B(T)\times[B/B]\ar[rr]^-{{\rm id}_{\B(T)}\times f}&&\B(T)\times (T\times_\car [G/G])}
$$
The second map is small but the first one is a $T$-torsor.
\bigskip

\begin{remark}As $f$ is a small resolution, the complex $f_!\Q$ is the intersection cohomology complex on $T\times_\car [G/G]$ and so it does not depend on the choice of $B$ and it is equipped with a canonical $W$-equivariant structure. This shows for instance that the induction of $W$-equivariant perverse sheaves on $T$ are equipped with a natural action of $W$ (and does not depend on $B$). Unfortunately this is not true anymore with $(q,p)_!\Q$, as we proved in \cite[\S 5.4]{LL} that this complex is not $W$-equivariant (this also follows from the results of \cite[\S 1.4]{Gun} in a more indirect way).
\label{perverse}\end{remark}

\subsection{The unipotent case}

As the diagram (\ref{triangle}) is above $\car$, it restricts to the unipotent elements and we get 

\begin{equation}
\xymatrix{&&[U/B]\ar[rrd]^{p_{\rm uni}}\ar[lld]_{q_{\rm uni}}\ar[d]^{(q_{\rm uni},p_{\rm uni})}&&\\
\B(T)&&\calS_{\rm uni}:=\B(T)\times [G_{\rm uni}/G]\ar[rr]^-{\pr_2}\ar[ll]_-{\pr_1}&&[G_{\rm uni}/G]}
\label{triangle}\end{equation}

The following result is the main theorem of \cite{LL}.

\begin{theorem}\cite[Theorem 1.1]{LL} The complex

$$
\calN_{\rm uni}:=\calN|_{\calS_{\rm uni}}=(q_{\rm uni},p_{\rm uni})_!\Q
$$
descends to a complex $\overline{\calN}_{\rm uni}$ on $\overline{\calS}:=\B(N)\times [G_{\rm uni}/G]$ (i.e. $\calN_{\rm uni}$ is equipped with a natural $W$-equivariant structure).
\label{maintheoLL}\end{theorem}

\begin{remark}In particular the complex $\calN_{\rm uni}$ does not depend on the choice of the Borel subgroup $B$ containing $T$.
\end{remark}

We then consider the cohomological correspondence

$$
\xymatrix{\B(N)&&\B(N)\times[G_{\rm uni}/G]\ar[rr]^-{\pr_2}\ar[ll]_-{\pr_1}&&[G_{\rm uni}/G]}
$$
with kernel $\overline{\calN}_{\rm uni}$.
\bigskip

We consider the associated induction and restriction functors

$$
\I_{\rm uni}:\calD_c^b(\B(N))\rightarrow\calD_c^b([G_{\rm uni}/G]),\hspace{.5cm}K\mapsto \pr_{2*}\underline{\rm Hom}\left(\overline{\calN}_{\rm uni},\pr_1^!(K)\right)
$$
and

$$
\R_{\rm uni}:\calD_c^b([G_{\rm uni}/G])\rightarrow\calD_c^b(\B(N)),\hspace{.5cm}K'\mapsto \pr_{1!}\left(\overline{\calN}_{\rm uni}\otimes\pr_2^*(K')\right).
$$
\bigskip

We also proved the following result.

\begin{theorem} \cite[\S 6.3, \S 7.1]{LL} The adjoint morphism 

$$
\R_{\rm uni}\circ\I_{\rm uni}\rightarrow 1
$$
is an isomorphism.

 If moreover $G$ is of type $A$ with connected center then the adjoint morphism $1\rightarrow \I_{\rm uni}\circ\R_{\rm uni}$ is also an isomorphism and so in this case
 
 \begin{equation}
 \calD_c^b(\B(N))\simeq \calD_c^b([G_{\rm uni}/G]).
\label{equiv} \end{equation}
\end{theorem}

The equivalence (\ref{equiv}) was previously proved by L. Rider \cite{Rider} by a different approach by interpreting $\calD_c^b(\B(N))$ as the derived category of finitely generated dg
modules over the smash product algebra $\Q[W]\# H^\bullet_G(G/B)$.

In the course of proving the above theorem, we proved the following one \cite[Theorem 6.2]{LL}.

\begin{theorem}We have a natural isomorphism

$$
\pr_{\B(N),\B(N)!}\left(\pr_{12}^*\overline{\calN}_{\rm uni}\otimes\pr_{23}^*\overline{\calN}_{\rm uni}\right)\simeq \Delta_{\B(N)\,!}\Q
$$
where $\Delta_{\B(N)}:\B(N)\rightarrow\B(N)\times\B(N)$ is the diagonal morphism,
$$
\pr_{\B(N),\B(N)}:\B(N)\times[G_{\rm uni}/G]\times \B(N)\rightarrow\B(N)\times\B(N)
$$
and $\pr_{12}$ (resp. $\pr_{23}$) is the projection on the first two coordinates 

$$
\B(N)\times[G_{\rm uni}/G]\times\B(N)\rightarrow \B(N)\times[G_{\rm uni}/G]
$$
(resp. on the last two coordinates).
\label{tech}\end{theorem}

\section{Geometrization of the orthogonality formulas}

\subsection{Reformulation of Theorem \ref{tech}}

Theorem \ref{tech} can be rephrased as follows.

\begin{theorem}We have an isomorphism 

\begin{equation}
(\R_{\rm uni}\times \R_{\rm uni})(\Delta_{[G_{\rm uni}/G]\, !}\Q)\simeq \Delta_{\B(N)\, !}\Q
\label{geoortho1}\end{equation}
where $\Delta_{[G_{\rm uni}/G]}:[G_{\rm uni}/G]\rightarrow[G_{\rm uni}/G]\times[G_{\rm uni}/G]$ denotes the diagonal morphism.
\label{maincoro}\end{theorem}

\begin{proof}We consider the following diagram

$$
\xymatrix{&\B(N)\times[G_{\rm uni}/G]\times\B(N)\ar[r]^-\gamma\ar[ld]_-\delta\ar[d]^\eta&[G_{\rm uni}/G]\ar[d]^{\Delta_{[G_{\rm uni}/G]}}\\
\B(N)\times\B(N)&\B(N)\times[G_{\rm uni}/G]\times\B(N)\times[G_{\rm uni}/G]\ar[r]^-\alpha\ar[l]_-\beta&[G_{\rm uni}/G]\times[G_{\rm uni}/G]}
$$
where $\eta=1_{\B(N)\times\B(N)}\times\Delta_{[G_{\rm uni}/G]}$ and the non-labelled maps are the obvious projections. Notice that the right square is cartesian.

We have (base change)

\begin{align*}
(\R_{\rm uni}\times\R_{\rm uni})(\Delta_{[G_{\rm uni}/G]\,!}\Q)&\simeq \beta_! \left(\alpha^*(\Delta_{[G_{\rm uni}/G]\,!}\Q)\otimes(\overline{\calN}_{\rm uni}\boxtimes\overline{\calN}_{\rm uni})\right)\\
&\simeq \beta_!\left(\eta_!\Q\otimes(\overline{\calN}_{\rm uni}\boxtimes\overline{\calN}_{\rm uni})\right).
\end{align*}

By the projection formula we thus get

\begin{align*}
(\R_{\rm uni}\times\R_{\rm uni})(\Delta_{[G_{\rm uni}/G]\,!}\Q)&\simeq \beta_!\circ \eta_!\circ \eta^*(\overline{\calN}_{\rm uni}\boxtimes \overline{\calN}_{\rm uni})\\
&\simeq \delta_!\circ\eta^*(\overline{\calN}_{\rm uni}\boxtimes \overline{\calN}_{\rm uni}).
\end{align*}
We conclude by Theorem \ref{tech}.

\end{proof}

We will  prove that the isomorphism of Theorem \ref{maincoro} is a geometric version of the first orthogonality formula for Green functions, namely if we assume that $T$ and $G$ are defined over a finite subfield $k_o$ of $k$, then the above complexes are naturally equipped with a Weil structure  and taking the characteristic functions yields the first orthogonality relation (\ref{ortho1}).
\bigskip

We introduce the functor $\P_{\rm uni}:=\I_{\rm uni}\circ\R_{\rm uni}:\calD_c^b([G_{\rm uni}/G])\rightarrow\calD_c^b([G_{\rm uni}/G])$. We will see that this a geometrization of the projector on uniform functions $\P^F$  which appear in the second orthogonality formula.   We will then prove that the isomorphism 

\begin{equation}
(\P_{\rm uni}\times\P_{\rm uni})(\Delta_{[G_{\rm uni}/G]\, !}\Q)\simeq (\I_{\rm uni}\times\I_{\rm uni})(\Delta_{\B(N)\, !}\Q)
\label{geoortho2}\end{equation}
obtained from (\ref{geoortho1}) by applying the functor $\I_{\rm uni}\times\I_{\rm uni}$ is a geometric version of the second orthogonality formula (\ref{ortho2}).
\bigskip

This will prove that the two orthogonality formulas are equivalent as we can recover the first one from the second one by applying $\R_{\rm uni}\times\R_{\rm uni}$ (using that $\R_{\rm uni}\circ\I_{\rm uni}\simeq 1$).

\subsection{Rational points of quotient stacks}

Assume first that $Z$ is a finite set on which a finite group $H$ acts on the right. Then we denote by $[Z/H]$ the category of $H$-equivariant maps $H\rightarrow Z$ for the right translation of $H$ on itself and let $\overline{[Z/H]}$ be the set of isomorphism classes of $[Z/H]$. A function on $[Z/H]$ is by definition a function on $\overline{[Z/H]}$ and we denote by $\calC([Z/H])$ the space of $\Q$-valued functions on $[Z/H]$. Notice that $\calC([Z/H])$ can be identified with the space of $\Q$-valued functions on $Z$ which are $H$-invariant. In particular $\calC([\bullet/H])=\Q$.
\bigskip

Now assume that $Z$ is a $k$-scheme endowed with an action of $k$-algebraic group $H$. We assume that $Z$, $H$ and the action of $H$ on $Z$ are defined over a finite subfield $k_o$ of $k$ with $k_o$-structures $Z_o$ and $H_o$. The quotient stack $[Z/H]$ is then also defined over $k_o$ and we denote by $F$ the corresponding geometric Frobenius on $Z$, $H$ and $[Z/H]$. 
\bigskip

A $k_o$-point of $[Z/H]$ is given by a cartesian diagram

$$\xymatrix{E_o\ar[rr]\ar[d]&&Z_o\ar[d]\\
{\rm Spec}(k_o)\ar[rr]^{\tau}&&[Z_o/H_o]}
$$
where the map $E_o\rightarrow{\rm Spec}(k_o)$ is an $H_o$-torsor and the map  $E_o\rightarrow Z_o$ is $H_o$-equivariant.
\bigskip

We denote by $[Z/H]^F$ the groupo\"id of $k_o$-points of $[Z/H]$.
\bigskip

Since the isomorphism classes of $H_o$-torsors on ${\rm Spec}(k_o)$ are parametrized by the set $H^1(F,H)=H^1(F,H/H^o)$ of $F$-conjugacy classes on $H$, we have the decomposition

$$
[Z/H]^F=\coprod_{[h]\in H^1(F,H)}[Z^{F\circ h}/H^{F\circ h}]
$$
where $h\in H$ denotes a representative of $[h]$ and $H^{F\circ h}=\{g\in H\,|\, F(h^{-1}gh)=g\}$.
\bigskip

If $H$ is connected then $H^1(F,H)$ is trivial and 

$$
[Z/H]^F=[Z^F/H^F].
$$

\subsection{Deligne-Lusztig induction on $[T/N]^F$}\label{DLu}

We assume that $G$, $T$, $B$  are defined over a finite subfield $k_o$ of $k$ and we denote by $F$ the Frobenius endomorphism on $G$ and $T$. 
\bigskip

Then $[G/G]^F=[G^F/G^F]$ and the space $\calC([G/G]^F)$ is just the space of functions $G^F\rightarrow\Q$ which are constant on conjugacy classes. 
\bigskip

We have 

$$
[T/N]^F=\coprod_{[n]\in H^1(F,N)}[T^{F\circ n}/N^{F\circ n}]
$$
and so

$$
\calC([T/N]^F)=\bigoplus_{[n]\in H^1(F,N)}\calC([T^{F\circ n}/N^{F\circ n}]).
$$
\begin{remark}Notice that 
$$
\calC([T/N]^F)=\calC([T/W]^F)=\bigoplus_{[w]\in H^1(F,W)}\calC([T^{F\circ w}/W^{F\circ w}]),
$$
and $\calC([T^{F\circ w}/W^{F\circ w}])$ can be identified with the subspace of $W^{F\circ w}$-invariant functions on $T^{F\circ w}$.
\end{remark}

Following \cite[\S 3.3]{LL1} we define the induction
$$
\I^F:\calC([T/N]^F)\rightarrow\calC([G/G]^F),\hspace{.5cm} f=(f_{[w]})_{[w]}\mapsto \sum_{[w]\in H^1(F,W)}\frac{1}{|W^{F\circ w}|} \, R_{T^{F\circ w}}^{G^F}(f_{[w]})
$$
where $f_w$ is the $[w]$-coordinate of $f$ and $R_{T^{F\circ w}}^{G^F}:\calC(T^{F\circ w})\rightarrow\calC(G^F)$ is the Deligne-Lusztig induction \cite{DLu}.
\bigskip

\begin{remark}Recall that, by definition,  the Green function $Q_{T_w^F}^{G^F}$ is the restriction of $R_{T_w^F}^{G^F}(1)$, where $1$ is the trivial character of $T_w^F$, to unipotent elements.
\end{remark}
\bigskip

We also define the restriction

$$
\R^F:\calC([G/G]^F)\rightarrow\calC([T/N]^F),\hspace{.5cm} h\mapsto \sum_{[w]\in H^1(F,W)}{^*}R^{G^F}_{T^{F\circ w}}(h)
$$
where ${^*}R^{G^F}_{T^{F\circ w}}$ is the Deligne-Lusztig restriction.
\bigskip

\begin{remark}From\cite[10.1.2]{DM} if $f$ is a unipotently supported function on $G_{\rm uni}^F$ then ${^*}R^{G^F}_{T_w^F}(f)$ is supported at $1\in T_w^F$ and

\begin{equation}
{^*}R^{G^F}_{T_w^F}(f)(1)=\frac{|T_w^F|}{|G^F|}\sum_{u\in G_{\rm uni}^F}Q_{T_w^F}^{G^F}(u)\, f(u).
\label{LR}\end{equation}

\end{remark}

Define $\calC([G/G]^F)_{\rm unif}$, called the space of uniform functions,  as the subspace of $\calC([G/G]^F)$ generated by the Deligne-Lusztig characters $R_{T^{F\circ w}}^{G^F}(\theta)$. 

Then $\P^F:=\R^F\circ \I^F:\calC([G/G]^F)\rightarrow\calC([G/G]^F)_{\rm unif}$ is the projector on uniform functions while $\R^F\circ\I^F:\calC([T/N]^F)\rightarrow\calC([T/N]^F)$ is the identity.

\subsection{Orthogonality relations}

In \cite[Theorem 6.2.3]{LL1}, we show that the maps $\I$ and $\R$ have a geometrical counterpart between categories of perverse sheaves but not at the level of derived categories because of Remark \ref{perverse}. However if we restrict ourselves to unipotent elements, then we are going to see in the next theorem that the maps

$$
\I^F_{\rm uni}:\calC(\B(N)^F)\rightarrow\calC([G_{\rm uni}/G]^F),\hspace{1cm}\R^F_{\rm uni}:\calC([G_{\rm uni}/G]^F)\rightarrow\calC(\B(N)^F)
$$
obtained from $\I^F$ and $\R^F$ by restriction to unipotently supported functions do correspond to the functors 

$$
\I_{\rm uni}:\calD_c^b(\B(N))\rightarrow\calC([G_{\rm uni}/G]),\hspace{1cm}\R_{\rm uni}:\calD_c^b([G_{\rm uni}/G])\rightarrow\calD_c^b(\B(N))
$$
defined earlier. 
\bigskip

\noindent Recall that

$$
\B(N)^F=\coprod_{[n]\in H^1(F,N)}[\bullet/N^{F\circ n}]
$$
and so $\calC(\B(N)^F)$ is just the space of $\Q$-valued functions on  $H^1(F,N)=H^1(F,W)$.

Then if we denote by $1_{[w]}$ the function that takes the value $1$ at $[w]$ and $0$ elsewhere then

$$
\I_{\rm uni}^F(1_{[w]})=Q_{T_{F(w)}^F}^{G^F}.
$$

If $\mathcal{Z}$ is an algebraic stack defined over $k_o$ with geometric Frobenius $F:\mathcal{Z}\rightarrow \mathcal{Z}$ then an $F$-equivariant complex on $\mathcal{Z}$ is a pair $(K,\varphi)$ with $K\in\calD_c^b(\mathcal{Z})$ and $\varphi:F^*K\simeq K$. We denote by $\calD_c^b(\mathcal{Z};F)$ the category of $F$-equivariant complexes on $\mathcal{Z}$ and we denote by ${\bf X}:\calD_c^b(\mathcal{Z};F)\rightarrow\calC(\mathcal{Z}^F)$ the map which sends $(K,\varphi)$ to its characteristic function $
{\bf X}_{(K,\varphi)}:x\mapsto \sum_i(-1)^i\Tr(\varphi^i_x, \calH^i_xK)$.
\bigskip

\begin{theorem}Then the following diagrams are commutative

$$
\xymatrix{\calD_c^b(\B(N);F)\ar[rr]^{\I_{\rm uni}}\ar[d]_{\bf X}&&\calD_c^b([G_{\rm uni}/G];F)\ar[d]^{\bf X}\\
\calC(\B(N)^F)\ar[rr]^{\I^F_{\rm uni}}&&\calC([G_{\rm uni}/G]^F)} \hspace{0.5cm}  \xymatrix{\calD_c^b([G_{\rm uni}/G];F)\ar[rr]^{\R_{\rm uni}}\ar[d]_{\bf X}&&\calD_c^b(\B(N);F)\ar[d]^{\bf X}\\
\calC([G_{\rm uni}/G]^F)\ar[rr]^{\R^F_{\rm uni}}&&\calC(\B(N)^F)}
$$
\label{comp}\end{theorem}

\begin{proof}This reduces essentially to the main result of \cite{LusGreen}. Denote by $\Ind_{\rm uni}:\calD_c^b(\B(T))\rightarrow\calD_c^b([G_{\rm uni}/G])$ the functor induced from the cohomological correspondence (\ref{triangle}) with kernel $\calN_{\rm uni}$ on $\calS_{\rm uni}$ (it is the restriction of ${\rm Ind}$ to unipotent elements).
\bigskip

As $\calN_{\rm uni}$ is $W$-equivariant by Theorem \ref{maintheoLL}, any $W$-equivariant structure $\theta$ on a complex $K$ on $\B(T)$ provides a group homomorphism $\theta^G:W\rightarrow{\rm Aut}(\Ind_{\rm uni}(K))$ and so the $W$-invariant part  $\Ind_{\rm uni}(K,\theta)^W$ of $\Ind_{\rm uni}(K)$ is well-defined \cite[\S 2.4]{LL}. Recall also that a $W$-equivariant complex $(K,\theta)$ on $\B(T)$ descends to a unique complex $\overline{K}$ on $\B(N)$ (i.e. $(K,\theta)$ is the pull back of $\overline{K}$ along the $W$-torsor $\B(T)\mapsto \B(N)$). Then \cite[Remark 2.17]{LL},

$$
\I_{\rm uni}(\overline{K})=\Ind_{\rm uni}(K)^W.
$$
If moreover $K$ is equipped with a Weil structure $\varphi:F^*(K)\simeq K$ compatible with the $W$-equivariant structure (see \cite[Diagram (6.1)]{LL1}), then $\overline{K}$, $\I_{\rm uni}(\overline{K})$, $\Ind_{\rm uni}(K)$ and $\Ind_{\rm uni}(K)^W$ are equipped with canonical Weil structures denoted abusively by $\varphi$ \cite[\S 6.1]{LL1} and by \cite[Formula (6.2)]{LL1} we have

$$
{\bf X}_{\I_{\rm uni}(\overline{K}),\varphi}={\bf X}_{\Ind_{\rm uni}(K)^W,\varphi}=\frac{1}{|W|}\sum_{w\in W}{\bf X}_{\Ind_{\rm uni}(K),\theta_w^G\circ\varphi}.
$$
Now the datum $(K,\theta,\varphi)$ provides for any $[w]\in H^1(F,W)$ an $F\circ w$-equivariant structure $\varphi_w:(F\circ w)^*K\simeq K$ and 

$$
{\bf X}_{\overline{K},\varphi}([w])={\bf X}_{K,\varphi_w}(\bullet)
$$
and so 

$$
\I_{\rm uni}^F\left({\bf X}_{\overline{K},\varphi}\right)=\frac{1}{|W|}\sum_{w\in W}R_{T^{F\circ w}}^{G^F}\left({\bf X}_{K,\varphi_w}\right).
$$
We are thus reduced to see that

$$
R_{T^{F\circ w}}^{G^F}\left({\bf X}_{K,\varphi_w}\right)={\bf X}_{\Ind_{\rm uni}(K),\theta_w^G\circ\varphi}.
$$
The $F\circ w$-equivariant complex $(K,\varphi_w)$ can be transfered to an $F$-equivariant complex $(K_w,\phi_w)$ on $\B(T_{F(w)})$ and we have

$$
R_{T^{F\circ w}}^{G^F}\left({\bf X}_{K,\varphi_w}\right)=R_{T_{F(w)}^F}^{G^F}\left({\bf X}_{K_w,\phi_w}\right).
$$
where  ${\bf X}_{K,\varphi_w}$ and ${\bf X}_{K_w,\phi_w}$ are regarded as functions on $T^{F\circ w}$ and $T_{F(w)}^F$ supported by $1$. Denote by ${\rm Ind}_{\rm uni}^w:\calD_c^b(\B(T_{F(w)}))\rightarrow\calD_c^b([G_{\rm uni}/G])$ the functor defined as ${\rm Ind}_{\rm uni}$ with $T_{F(w)}$ instead of $T$. Then canonically

$$
{\rm Ind}_{\rm uni}^w(K_w,\phi_w)\simeq \left({\rm Ind}_{\rm uni}(K),\theta^G_w\circ\varphi\right)
$$
and so 

$$
{\bf X}_{{\rm Ind}_{\rm uni}^w(K_w,\phi_w)}={\bf X}_{{\rm Ind}_{\rm uni}(K),\theta^G_w\circ\varphi}.
$$
We are thus reduced to prove that

\begin{equation}
R_{T_{F(w)}^F}^{G^F}\left({\bf X}_{K_w,\phi_w}\right)={\bf X}_{{\rm Ind}_{\rm uni}^w(K_w,\phi_w)}.
\label{Lus}\end{equation}
If $K$ is the constant sheaf $\Q$, then on the right hand side this the Green function $Q_{T_{F(w)}^F}^{G^F}$ of Deligne-Lusztig and on the left hand side this is the geometric Green functions defined at first by Springer \cite{Sp} (in the Lie algebra setting) when the characteristic of $k$ is \emph{good} and by Lusztig \cite[Section 3]{LUG} in any characteristic. This formula was first proved by Kazhdan \cite{Ka} when the characteristic is large enough using Springer's definition, and by Lusztig in full generality in \cite{LusGreen} (notice that in this case $K$ is a perverse sheaf on $\B(T)$ and the category of perverse sheaves on $\B(T)$ is equivalent to that on a point and we are thus in the context of Lusztig's paper). 
\bigskip

Denote by $N_{\rm uni}^w$ the characteristic function of the complex $\calN_{\rm uni}^w$ on $\B(T_{F(w)})\times[G_{\rm uni}/G]$ (defined as $\calN_{\rm uni}$ but with $T_{F(w)}$ instead of $T$). Then it follows from (\ref{Lus}) applied to $K=\Q$ that

\begin{align*}
Q_{T_{F(w)}^F}^{G^F}=N_{\rm uni}^w(1,\cdot)
\end{align*}
Going back to an arbitrary complex $K$ we have, for all $x\in G_{\rm uni}^F$

\begin{align*}
{\bf X}_{{\rm Ind}_{\rm uni}^w(K_w,\phi_w)}(x)&=N_{\rm uni}^w(1,x)\, {\bf X}_{K_w,\phi_w}(1)\\
&=Q_{T_{F(w)}^F}^{G^F}(x) \,{\bf X}_{K_w,\phi_w}(1)\\
&=R_{T_{F(w)}^F}^{G^F}\left({\bf X}_{K_w,\phi_w}(1)\cdot 1_{\{1\}}\right)(x)\\
&=R_{T_{F(w)}^F}^{G^F}\left({\bf X}_{K_w,\phi_w}\right)(x).
\end{align*}
\end{proof}

Notice that the characteristic functions of $(\R_{\rm uni}\times\R_{\rm uni})(\Delta_{[G_{\rm uni}/G]\, !}\Q)$ and $\Delta_{\B(N)\, !}\Q$ are functions on $H^1(F,W)\times H^1(F,W)$.

\begin{theorem}Taking the value of the characteristic functions of
\begin{equation}
(\R_{\rm uni}\times\R_{\rm uni})(\Delta_{[G_{\rm uni}/G]\, !}\Q)\simeq\Delta_{\B(N)\, !}\Q
\label{ortho11}\end{equation}
at $([w],[w'])\in H^1(F,W)\times H^1(F,W)$ yields the first orthogonality relation

\begin{equation}
\frac{1}{|G^F|}\sum_{u\in G_{\rm uni}^F}Q_{T_w^F}^{G^F}(u)Q_{T_{w'}^F}^{G^F}(u)=\frac{|W^{wF}|}{|T_w^F|}\, \delta_{[w|,[w']}.
\label{ortho21'}\end{equation}
\end{theorem}

\begin{proof} The characteristic function of $\Delta_{[G_{\rm uni}/G]\, !}\Q$ is the function $\Delta_{\rm uni}^F:G_{\rm uni}^F\times G_{\rm uni}^F\rightarrow\Q$ defined by

$$
\Delta_{\rm uni}^F(u,v)=\begin{cases}|C_G(u)^F|&\text{ if } u \text{ if } G^F-\text{conjugate to }v,\\0&\text{ otherwise.}\end{cases}
$$
The characteristic function of $\Delta_{\B(N)\,!}\Q$ is the function $\Delta_{\B(N)^F\, !}(1)$ where 

$$
\Delta_{\B(N)^F}:\B(N)^F\rightarrow\B(N)^F\times\B(N)^F
$$
is the diagonal map. We have, for $[n],[n']\in H^1(F,N)$,

\begin{equation}
\Delta_{\B(N)^F\, !}(1)([n],[n'])=\begin{cases}|N^{F\circ n}|&\text{ if }[n]=[n'],\\
0&\text{ otherwise.}\end{cases}
\label{!}\end{equation}
From Theorem \ref{comp} we are therefore reduced to prove that

$$
(\R_{\rm uni}^F\times \R_{\rm uni}^F)(\Delta_{\rm uni}^F)([n],[n'])=\frac{|T^{F\circ n}|\, |T^{F\circ n'}|}{|G^F|}\sum_{u\in G_{\rm uni}^F}Q_{T_{F(w)}^F}^{G^F}(u)Q_{T_{F(w')}^F}^{G^F}(u)
$$
where $w$ and $w'$ are the images of $n$ and $n'$ in $W$ (notice that $T^{F\circ n}\simeq T_{F(w)}^F$ and so $Q^{G^F}_{T^{F\circ n}}=Q_{T_{F(w)}^F}^{G^F}$).

We have 

$$
(\R^F_{\rm uni}\times \R^F_{\rm uni})(\Delta_{\rm uni}^F)([n],[n'])=\left({^*}R^{G^F}_{T^{F\circ n}}\times {^*}R^{G^F}_{T^{F\circ n'}}\right)(\Delta_{\rm uni}^F)(1,1)
$$
where by notation abuse we denote by $\Delta_{\rm uni}^F:G^F\times G^F\rightarrow\Q$ the extension by $0$ of $\Delta_{\rm uni}^F:G_{\rm uni}^F\times G_{\rm uni}^F\rightarrow\Q$. 

It follows from Formula (\ref{LR}) applied to the group $G\times G$ instead of $G$, that

\begin{align*}
\left({^*}R^{G^F}_{T^{F\circ n}}\times {^*}R^{G^F}_{T^{F\circ n'}}\right)(\Delta_{\rm uni}^F)(1,1)&=\frac{|T^{F\circ n}|\,|T^{F\circ n'}|}{|G^F|^2}\sum_{u,v\in G_{\rm uni}^F}Q_{T_{F(w)}^F}^{G^F}(u)Q_{T_{F(w')}^F}^{G^F}(v)\,\Delta(u,v)\\
&=\frac{|T^{F\circ n}|\,|T^{F\circ n'}|}{|G^F|}\sum_{u\in G_{\rm uni}^F}Q_{T_{F(w)}^F}^{G^F}(u)Q_{T_{F(w')}^F}^{G^F}(u).
\end{align*}
We conclude by noticing that $|N^{F\circ n}|=|T_{F(w)}^F|\, |W^{F(w)F}|$ and $|T^{F\circ n}|=|T_{F(w)}^F|$.
\end{proof}

Let us now discuss the second orthogonality formula.

\begin{theorem}Put $\P_{\rm uni}:=\I_{\rm uni}\circ\R_{\rm uni}$. Taking the value of characteristic functions of

\begin{equation}
 (\I_{\rm uni}\times\I_{\rm uni})(\Delta_{\B(N)\, !}\Q)\simeq(\P_{\rm uni}\times\P_{\rm uni})(\Delta_{[G_{\rm uni}/G]\, !}\Q)
\label{ortho22}\end{equation}
at $(u,v)\in G_{\rm uni}^F\times G_{\rm uni}^F$, yields the second orthogonality formula

\begin{equation}
\frac{1}{|W|}\sum_{w\in W}|T_w^F|\,Q_{T_w^F}^{G^F}(u)\,Q_{T_w^F}^{G^F}(v)=(\P^F_{\rm uni}\times \P^F_{\rm uni})(\Delta_{\rm uni}^F)(u,v)
\label{ortho22'}\end{equation}
where $\P^F_{\rm uni}:=\I_{\rm uni}^F\circ\R_{\rm uni}^F$.
\end{theorem}


\begin{proof}First of all, by Theorem \ref{comp}, we see that the characteristic function of the right hand side of (\ref{ortho22}) is the right hand side of (\ref{ortho22'}). 
\bigskip

By Theorem \ref{comp}, the characteristic function of the right hand side of (\ref{ortho21'}) is the function

$$
(\I_{\rm uni}^F\times\I_{\rm uni}^F)(\Delta_{\B(N)^F\, !}(1)).
$$
By (\ref{!}), the function $\Delta_{\B(N)^F\, !}(1)\in\calC(\B(N)^F\times\B(N)^F)=\bigoplus_{[n],[n']\in H^1(F,N)}\Q$ decomposes as
$$
\Delta_{\B(N)^F\, !}(1)=\sum_{[n]\in H^1(F,N)}|N^{F\circ n}|.
$$

For $(u,v)\in G_{\rm uni}^F\times G_{\rm uni}^F$ we have
$$
 (\I_{\rm uni}^F\times\I_{\rm uni}^F)(\Delta_{\B(N)^F\,!}(1))(u,v)=\sum_{[w]\in H^1(F,W)}\frac{|N^{F\circ n}|}{|W^{F\circ w}|\, |W^{F\circ w'}|}\left(R_{T_{F(w)}^F}^{G^F}\times R_{T_{F(w)}^F}^{G^F}\right)(1_{\{1\}})(u,v)
$$
where $1_{\{1\}}$ is the function on $T_{F(w)}^F$ that takes the value $1$ at $1$ and $0$ elsewhere and where $n$ is a representative of $[w]$ in $N$.

Therefore

\begin{align*}
 (\I_{\rm uni}^F\times\I_{\rm uni}^F)(\Delta_{\B(N)^F\,!}(1))(u,v)&=\sum_{[w]\in H^1(F,W)}\frac{|N^{F\circ n}|}{|W^{F\circ w}|\, |W^{F\circ w}|}Q_{T_{F(w)}^F}^{G^F}(u)\, Q_{T_{F(w)}^F}^{G^F}(v)\\
 &=\sum_{[w]\in H^1(F,W)}\frac{|T^{F\circ n}|}{|W^{F\circ w}|}Q_{T_{F(w)}^F}^{G^F}(u)\, Q_{T_{F(w)}^F}^{G^F}(v)\\
 &=\sum_{[w]\in H^1(F,W)}\frac{|T_w^F|}{|W^{wF}|}Q_{T_w^F}^{G^F}(u)\, Q_{T_w^F}^{G^F}(v)\\
 &=\frac{1}{|W|}\sum_{w\in W}|T_w^F|Q_{T_w^F}^{G^F}(u)\, Q_{T_w^F}^{G^F}(v).
\end{align*}
\end{proof}

We can see that the isomorphism (\ref{ortho11}) and (\ref{ortho22}) are equivalent using the functors $\I_{\rm uni}\times\I_{\rm uni}$ and $\R_{\rm uni}\times\R_{\rm uni}$. Hence we deduce the following result.

\begin{corollary}The two orthogonality formulas (\ref{ortho21'}) and (\ref{ortho22'}) are equivalent.
\end{corollary}

\end{document}